\documentclass{amsart}
\usepackage{amsthm}
\usepackage{amsmath,amssymb}
\newtheorem{theorem}{Theorem}

\newtheorem{prop}[theorem]{Proposition}
\newtheorem{lem}[theorem]{Lemma}
\newtheorem{definition}[theorem]{Definition}

\numberwithin{theorem}{section}
\newtheorem{rem}[theorem]{Remark}
\allowdisplaybreaks
\newcommand{\Om} {\Omega}

\newcommand{\be} {\begin{equation}}
\newcommand{\ee} {\end{equation}}
\newcommand{\bea} {\begin{eqnarray}}
\newcommand{\eea} {\end{eqnarray}}
\newcommand{\Bea} {\begin{eqnarray*}}
\newcommand{\Eea} {\end{eqnarray*}}

\newcommand{\al} {\alpha}

\newcommand{\la} {\lambda}

\def\R{{\mathbb R}}

\def\C{{\mathcal C}}

\def\C{{\mathcal C}}

\newcommand{\rar}{\rightarrow}
\numberwithin{equation}{section}

\begin{document}
\title[Bifurcation theorems] { On the bifurcation for fractional Laplace equations}

\author[G.Dwivedi,\, J.\,Tyagi, R.B.Verma ]
{ G. Dwivedi,  J.Tyagi, R.B.Verma }

\address{G.\,Dwivedi \hfill\break
 Indian Institute of Technology Gandhinagar \newline
 Vishwakarma Government Engineering College Complex \newline
 Chandkheda, Visat-Gandhinagar Highway, Ahmedabad \newline
 Gujarat, India - 382424}
 \email{dwivedi\_gaurav@iitgn.ac.in}

\address{J.\,Tyagi \hfill\break
 Indian Institute of Technology Gandhinagar \newline
 Vishwakarma Government Engineering College Complex \newline
 Chandkheda, Visat-Gandhinagar Highway, Ahmedabad \newline
 Gujarat, India - 382424}
 \email{jtyagi@iitgn.ac.in, jtyagi1@gmail.com}

\address{R.B.Verma \hfill\break
 Indian Institute of Technology Gandhinagar \newline
 Vishwakarma Government Engineering College Complex \newline
 Chandkheda, Visat-Gandhinagar Highway, Ahmedabad \newline
 Gujarat, India - 382424}
 \email{ram.verma@iitgn.ac.in}

\date{14--06--2016}
\thanks{Submitted 14--06--2016.  Published-----.}

\begin{abstract}
In this paper, we consider the bifurcation problem for fractional Laplace equation
 \begin{eqnarray*}
 \begin{array}{ll}
 (-\Delta)^{s} u =  \lambda u + f(\lambda,\,x,\,u)& \mbox{in }\Omega,\\
 u   = 0 &\mbox{in }\R^n\backslash \Omega,
\end{array}
 \end{eqnarray*}
where $\Om\subset \R^n,\,n> 2s (0<s<1)$ is an open bounded subset with smooth boundary,\,$(-\Delta)^{s}$ stands for the fractional Laplacian.
We show that a continuum of solutions bifurcates out from the principal
eigenvalue $\lambda_1$ of the eigenvalue problem
\begin{eqnarray*}
 \begin{gathered}
 (-\Delta)^{s} v  =  \lambda v\,\,\,\mbox{in}\,\,\Omega,\\
 v   = 0 \,\,\,\,\mbox{in}\,\,\,\,\R^n \backslash\Omega,
\end{gathered}
 \end{eqnarray*}
and, conversely.

\end{abstract}
\maketitle
\section{Introduction}
In this paper, we consider the following bifurcation problem
 \begin{eqnarray}\label{bif}
 \left\{
\begin{gathered}
 (-\Delta)^{s} u =  \lambda u + f(\lambda,\,x,\,u)\,\,\,\mbox{in}\,\,\Omega,\\
 u   = 0 \,\,\,\,\mbox{in}\,\,\,\,\,\,\,\R^n\backslash \Omega,
\end{gathered}\right.
 \end{eqnarray}
where $\Om\subset \R^n,\,n> 2s \,(0<s<1)$ is an open bounded subset with smooth boundary, $(-\Delta)^{s}$ stands for the fractional Laplacian
and the conditions on $f$ will be specified later.
The fractional Laplacian is the infinitesimal generator of L\'{e}vy stable diffusion process, see\,\cite{ap,be} and it appears in several branches of sciences
and engineering and has substantial applications such as free boundary problem \cite{cs1,cs2},  thin film obstacle problem\, \cite{cs2,sil}, conformal geometry \cite{cha},\,
cell biology\,\cite{sa}, chemical and contaminant transport in heterogeneous aquifers \cite{ben}, phase transitions \cite{sir}, crystal dislocation\,\cite{ga} and many others.
We refer to \cite{di} and references therein for an elementary introduction on this subject.

Recently, there has been a good amount of work on the existence and qualitative questions to fractional Laplace equations. There are several articles dealing with these questions.
We just name a few articles, see for instance\,\cite{cap,ta1,ta2} and the references therein. We refer to \cite{vaz} on recent progress in the
study of nonlinear diffusion equations involving nonlocal, long-range diffusion effects. Using variational and topological techniques, more precisely
the variational principle of Ricceri, the authors \cite{bis} obtained the existence of a nontrivial solution to the problem similar to \eqref{bif}.
Recently, using pseudo-index related to the $\mathbf{Z}_2$-cohomological
index theory, Perera et al. \cite{per} established the bifurcation and multiplicity results for critical fractional $p$-laplacian problems.

So in this context, it is natural to ask the bifurcation results to \eqref{bif} with different delicate techniques.
In the best of our knowledge, we are not aware with the bifurcation results for \eqref{bif}.  We recall that there is an enormous work on bifurcation theory
to Laplace as well as $p$-Laplace equations and other higher order operators. We just mention those articles which are closely related to this paper.

Using the topological degree argument and fixed point theory, Rumbos and Edelson\,\cite{ram}
studied the bifurcation from the first eigenvalue in $\R^N$ to semilinear elliptic
equations and obtained the existence of bifurcating branches. Using the topological degree argument,  Dr\'{a}bek and Huang \cite{dra}, Dr\'{a}bek\,\cite{dra1,dra2}
obtained the bifurcating branches to
$p$-Laplace equations. We remark that the isolatedness of the principal eigenvalue $\lambda_1$ of the associated eigenvalue problem with bifurcation problem
plays an important role to obtain the bifurcation results to Laplace, $p$-Laplace and fully nonlinear elliptic equations, see for instance\,\cite{ram} for
Laplace equations,\,\cite{dra} for $p$-Laplace equations and \cite{mar,beq} for fully nonlinear elliptic equations.
In case of Laplace as well as $p$-Laplace equations, the isolatedness of $\lambda_1$ is obtained by establishing an
estimate on the nodal domains of the solutions to the eigenvalue problem and in fact, Picone's identity plays a role to get an estimate on the nodal domains.
One of the difficulties here is to obtain the isolatedness of the principal eigenvalue $\lambda_1$ of
\begin{eqnarray}\label{ei}
 \left\{
\begin{gathered}
 (-\Delta)^{s} v  =  \lambda v\,\,\,\mbox{in}\,\,\Omega,\\
 v   = 0 \,\,\,\,\mbox{in}\,\,\,\,\,\,\,\R^n \backslash\Omega.
\end{gathered}\right.
 \end{eqnarray}
Indeed, we obtain an elementary inequality which helps us to get an estimate on the nodal domains and which further leads to the isolatedness of $\lambda_1$ of \eqref{ei}.
We make use of the isolatedness of $\lambda_1$ and by the celebrated Rabinowitz's global bifurcation theorem \cite{rab}, prove the bifurcation results to \eqref{bif}.
We also show that if $(\bar{\lambda},\,0)$ is a bifurcation point of \eqref{bif}, then $\bar{\lambda}$ is an eigenvalue of \eqref{ei}. A similar approach is used by
Del Pino and Man\'{a}sevich \cite{del} and Busca et al.\cite{beq}, where the authors have obtained the similar results for $p$-Laplace equations and the equations
involving Pucci's operators, respectively.

More precisely, we state the following theorems, which we will prove in the last section, using the below hypotheses:

\begin{theorem}
 Let $2s<n<4s,\,\,0<s<1.$ Let \rm{(H1)}--\rm{(H3)} hold. Then the principal eigenvalue $\lambda_1$ of the eigenvalue problem \eqref{ei} is a bifurcation point of \eqref{bif}.
 \end{theorem}

\begin{theorem}
 Let \rm{(H1)} and \rm{(H4)} hold. Let $(\bar{\lambda},\,0)$ be a bifurcation point of \eqref{bif}, then $\bar{\lambda}$ is an eigenvalue of \eqref{ei}.
\end{theorem}

We list the following hypotheses which have been used in this paper:
\begin{itemize}
 \item[(H1)]\,\,$f: \R\times\Om\times \R\longrightarrow \R$ such that $f$ is a Carath\'{e}odory function, i.e., $f(\cdotp,\,x,\,\cdotp)$ is continuous for a.e. $x\in \Om$ and $f(\lambda, \cdotp,\,t)$ is measurable
 for all $(\lambda,\,t)\in \R^2.$
 \item[(H2)]\,\,$|f(\lambda,\,x,\,t)|\leq C(\lambda)(m_{1}(x) + m_{2}(x)|t|^\gamma )$ for a.e. $x\in \Om, t\in \R$ and $0\leq C(\lambda)$ is continuous and bounded on bounded subsets of $\R,
 1<\gamma < 2_{*}(s)-1,\,\,m_1 \in L^{\frac{n}{2s}}(\Om),\,\,0\leq m_{2}\in L^{\gamma_1}(\Om), $ where $2_{*}(s)= \frac{2n}{n-2s},\,\,\gamma_1= \frac{2n}{n+2s- (n-2s)\gamma},\,\,n>2s.$

 \item[(H3)]\,\,$\lim_{t\rightarrow 0} \frac{f(\lambda,\,x,\,t)}{t}=0$\,\,uniformly for a.e. $x\in \Om$ and $\lambda$ in a bounded interval.

 \item[(H4)]\,\,There is a $q$ with $1<q<2_{*}(s),$ such that
 $$  \lim_{|t|\rightarrow \infty}  \frac{|f(\lambda,\,x,\,t)|}{|t|^{q-1} } =0                            $$
 uniformly a.e.\,with respect to $x$\, and uniformly with respect to $\lambda$ on bounded sets.
 \end{itemize}
The organization of this paper is as follows. Section 2 deals with the preliminaries on the fractional Laplacian. In Section 3, we establish auxiliary lemmas and propositions which
have used to establish the main theorems. Section 4 deals with the introduction of topological degree and its properties and finally we prove the main theorems in Section 5.

 \section{Preliminaries}
Let us recall the brief preliminaries on the fractional Laplacian. Let $0<s<1.$ There are various definitions to define fractional Laplacian $(-\Delta)^{s} u$ of a function $u$
defined on $\R^n.$
It is well known that $(-\Delta)^{s}$ on $\R^n,\,\,0<s<1$ is a nonlocal operator and it can be defined through its Fourier transform.
Thus, if $u$ is a function in the Schwarz class in $\R^n,\,n\geq 1,$ denoted by $S(\R^n),$  then we have
$$  \widehat{(-\Delta)^{s} }u(\xi)=    {|\xi|}^{2s} \widehat{u} (\xi)   $$ and therefore
we write $$(-\Delta)^{s} u=f  \,\,\,\mbox{ if }\,\,\,\,\widehat{f}(\xi) = {|\xi|}^{2s} \widehat{u} (\xi),$$
where\, $\widehat{}$\, is Fourier transform. When $u$ is sufficiently regular, the fractional Laplacian of a function $u:\R^n\longrightarrow \R$ is defined as follows:

$$   (-\Delta)^{s} u(x)  = C_{n,\,s} P.V.\int_{\R^n} \frac{u(x)- u(y)}{|x-y|^{n+2s}} dy=
C_{n,\,s}\lim_{\epsilon \rightarrow 0} \int_{\R^n \backslash B_{\epsilon}(x)} \frac{u(x)- u(y)}{|x-y|^{n+2s}} dy,  $$
where $$C_{n,\,s}= s.2^{2s} \frac{\Gamma(s+\frac{n}{2} )}{\pi^{\frac{n}{2} } \Gamma(1-s)},$$ which is a normalization constant.

One can also write the above singular integral as follows:
\begin{equation}
 (-\Delta)^{s} u(x) = - \frac{C_{n,\,s}}{2}\int_{\R^n} \frac{u(x+y) + u(x-y)- 2 u(x) }  {|y|^{n+2s} } dy,\,\,\forall\,\,x\in \R^n,\,u\in S(\R^n),\,
\end{equation}
see\,\cite{di}. When $s<\frac{1}{2}$ and $f\in C^{0,\,\alpha} (\R^n)$ with $\alpha >2s,$ or if $f \in C^{1,\,\alpha}(\R^n),\,\,1+ 2\alpha >2s,$ the above integral is well-defined.
By the celebrated work of Caffarelli and Silvestre \cite{cs}, the nonlocal operator can be expressed as a generalized Dirichlet-Neumann map for a certain elliptic boundary
value problem with nonlocal operator defined on the upper half space $\R_{+}^{n+1}:= \{ (x,\,y)|\,\,x\in \R^n,\,y>0    \},$ i.e., for a function $u:\R^n\longrightarrow \R,$
we consider the extension $v(x,\,y):\,\R^n\times [0,\,\infty)\longrightarrow \R$ that satisfy
\begin{eqnarray}
 v(x,\,0)= u(x).\nonumber\\
 \Delta_{x}v + \frac{1-2s}{y} v_{y} + v_{yy}=0.\label{ext}
 \end{eqnarray}
\eqref{ext} can  also be written as
\begin{equation}
 div(y^{1-2s} \nabla v )=0,
\end{equation}
which is the Euler-Lagrange equation for the functional
\begin{equation}
 J(v)= \int_{y>0} |\nabla v|^{2} y^{1-2s}dxdy.
 \end{equation}
Then it can be seen from \cite{cs}, that
\begin{equation}
 C (-\Delta)^{s} u = \lim_{y\rightarrow 0^+} - y^{1-2s} v_{y}= \frac{1}{2s} \lim_{y\rightarrow 0^+} \frac{v(x,\,y)- v(x,\,0)}{y^{2s}},
\end{equation}
where $C$ is some constant depending on $n$ and $s.$
 The space $H^{s} (\R^n)= W^{s,\,2}(\R^n)$ is defined by
\begin{align*}
 H^{s}(\R^n)& = \left\{  u\in L^{2}(\R^n):\,\,\frac{|u(x) -u(y)|} {|x-y|^{s+\frac{n}{2}}} \in L^{2} (\R^n \times \R^n) \right \}  \\
&= \left\{   u\in L^{2}(\R^n):\,\,\,\,(1+|\xi|^2)^{\frac{s}{2}} \widehat{u} (\xi)\in L^{2} (\R^n)\right \}
\end{align*}
and the norm is defined as follows:
\begin{equation*}
 || u||_{s}:= ||u||_{H^{s}(\R^n)}  = \left(\int_{\R^n} \int_{\R^n} \frac{|u(x) -u(y)|^2} {|x-y|^{n+2s}} dxdy  + \int_{\R^n}|u|^{2} dx \right)^{\frac{1}{2} }.
\end{equation*}
The term
\begin{equation*}
 [u]_{s}:=  \left(\int_{\R^n} \int_{\R^n} \frac{|u(x) -u(y)|^2} {|x-y|^{n+2s}} dxdy \right)^{\frac{1}{2} }
\end{equation*}
is called the Gagliardo norm of $u.$ We recall that $H^{s}(\R^n)$ is a Hilbert space and its norm is induced by the inner product

\begin{equation*}
 <u,\,v >_{H^{s}(\R^n)}  = \int_{\R^n} \int_{\R^n} \frac{(u(x) -u(y))(v(x) - v(y))} {|x-y|^{n+2s}} dxdy  + \int_{\R^n} u(x) v(x) dx.
 \end{equation*}

Let $\{ \phi_{k}\}$ be an orthonormal basis of $L^{2}(\Omega)$ with $||\phi_{k}||_{2}=1$
forming a spectral decomposition of $-\Delta$ in $\Omega$ with zero Dirichlet boundary conditions and $\lambda_{k}$ be the corresponding eigenvalues. Let
\begin{equation*}
 H_{0}^{s}(\Omega)= \left\{ u= \sum_{k=0}^{\infty}a_{k} \phi_{k} \in L^{2}(\Omega) \,\,:\,\,  || u ||_{ H_{0}^{s}(\Omega)}=
 \left( \sum_{k=0}^{\infty} a_{k}^2 \lambda_{k}^s \right)^{\frac{1}{2} } <\infty   \right\}.
\end{equation*}

We denote $H^{-s}(\Omega)$ the dual space of $H_{0}^{s}(\Omega).$ For $u\in H_{0}^{s}(\Omega),\,\,\,u=\sum_{k=0}^{\infty}a_{k} \phi_{k} $
with $a_k= \int_{\Omega} u \phi_{k} dx,$ then one can also define $(-\Delta)^s$ as follows (see\,\cite{di,ros}) :

\begin{equation*}
 (-\Delta)^{s} u= \sum_{k=0}^{\infty}a_{k} \lambda_{k}^{s}\phi_{k}\in H^{-s}(\Omega).
 \end{equation*}
We call $\{(\phi_k,\,\lambda_{k}^s)\}$ the eigenfunctions and eigenvalues of $ (-\Delta)^{s}$ in $\Omega$ with zero Dirichlet boundary conditions.
We recall that these two definitions of fractional Laplacian give rise to two completely different non-local operators of fractional type, i.e., the fractional Laplacian can be defined
in both these ways, but the two definitions are not equivalent in bounded domains, see for instance \cite{bran,mus}. In this paper, we will be using the definition of
fractional Laplacian via Fourier transform or equivalently, as an integral in the P.V. sense.

Let $\lambda_1>0$ be the first eigenvalue of $(-\Delta)^{s}$ in $\Omega$ and $\phi_1 >0$ be the corresponding eigenfunction (first eigenfunction), i.e.,
\begin{eqnarray*}
 \left\{
\begin{gathered}
 (-\Delta)^{s} \phi_1  =  \lambda_1 \phi_1\,\,\,\mbox{in}\,\,\Omega,\\
 \phi_1   > 0 \,\,\,\,\mbox{in}\,\,\,\,\,\,\Omega,\\
 \phi_1   = 0 \,\,\,\,\mbox{in}\,\,\,\,\,\,\,\R^n\backslash \Omega.
\end{gathered}\right.
 \end{eqnarray*}
The variational characterization of $\lambda_1$ is given by
\begin{equation}\label{vari}
 \lambda_1= \inf\left\{ \int_{\R^n}|(-\Delta)^{\frac{s}{2}}v|^2 dx :\,\,\,v\in X=H_{0}^{s}(\Om)\,\,\,\mbox{and}\,\,\int_{\Om}v^2 dx=1\right\}.
\end{equation}
The next section deals with the auxiliary lemmas and
propositions, which have been used in the proof of main theorems.

\section{Auxiliary Lemmas and Propositions}

 Let us denote $H_{0}^{s}(\Om)$ by $X$ and its dual $H^{-s}(\Om)$ by $X^*.$ We define the operators $S,\,T,\,F(\lambda,\,\cdotp):\,\,X\longrightarrow X^*$ as follows:\\
 For\,\,$u,\,v\in X,$
 \begin{eqnarray}
  ((S(u),\,v))_{X}= \int_{\R^n} (-\Delta)^{\frac{s}{2} } u\cdotp (-\Delta)^{\frac{s}{2}}v,\\
  ((T(u),\,v))_{X}=\int_{\Om}  u\cdotp v,\\
  ((F(\lambda,\,u),\,v))_{X}=\int_{\Om} f(\lambda,\,x,\,u)\cdotp v.
 \end{eqnarray}
Let us recall the following embeddings.
\begin{theorem}\cite{di}\label{em}
 The following embeddings are continuous:\\
 \rm{(1)}\,\,$H^{s}(\R^n)\hookrightarrow L^{q}(\R^n),\,\,2\leq q\leq \frac{2n}{n-2s},\,\,\,\mbox{if}\,\,\,n>2s,$\\
 \rm{(2)}\,\,$H^{s}(\R^n)\hookrightarrow L^{q}(\R^n),\,\,2\leq q\leq \infty,\,\,\,\mbox{if}\,\,\,n=2s,$\\
 \rm{(3)}\,\,$H^{s}(\R^n)\hookrightarrow C_{b}^{j}(\R^n),\,\,\mbox{if}\,\,\,n<  2(s-j).$\\
 Moreover, for any $R>0$ and any $p\in [1,\,2_{*}(s))$ the embedding $H^{s}(B_R)   \hookrightarrow\hookrightarrow L^{p}(B_R)$
 is compact, where
 $$ C_{b}^{j}(\R^n)= \{ u\in C^{j}(\R^n):  D^{k} u \,\,\mbox{is bounded on }\,\R^n\,\,\mbox{for}\,|k|\leq j\}. $$

 \end{theorem}

\begin{lem}\label{lmm1}
 The operators $S,\,T,\,F$ are well-defined, $S$ and $T$ are continuous and $F$ satisfies
\begin{equation}\label{li}
  \lim_{||u||_X \rightarrow 0} \frac{||F(\lambda,\,u) ||_{X^*} }   {||u||_{X}}=0
\end{equation}
uniformly for $\lambda$ in a bounded subset of $\R.$
 \end{lem}
\begin{proof}
In the proof of this lemma, $c_i,\,\,i=1,\,2,\cdots,\,6$ are positive constants.
It is easy to see that the  operators $S,\,T,\,F$ are well-defined and $S$ and $T$ are continuous.
Now we show \eqref{li}. By definition,

\begin{equation} \label{1}
\begin{aligned}
\lim_{||u||_{X} \rightarrow 0} \frac{||F(\lambda,\,u) ||_{X^*} }   {||u||_{X}} &= \lim_{||u||_X \rightarrow 0} \,\sup_{||v||_{X}\leq 1} \frac{1}{||u||_{X}}
\left |\int_{\Om} f(\lambda,\,x,\,u)v \right | \\
& \leq \lim_{||u||_{X} \rightarrow 0} \,\sup_{||v||_{X}\leq 1} \int_{\Om} \frac{|f(\lambda,\,x,\,u)| |v | |\tilde{u}|}{  |u|},\,\,\,\mbox{where}\,\,\tilde{u}= \frac{u}{||u||_X}.
\end{aligned}
\end{equation}
For $\delta>0,\,$ let us define $\Om_{\delta}(u)= \{x\in \Om: |u|\geq \delta \}.$ We claim  that, as $||u||_{X}\rightarrow 0,\,\,\,\,\mbox{meas}(\Om_{\delta}(u))\rightarrow 0.$
Assume, on contrary, suppose that $\mbox{meas}(\Om_{\delta}(u))\geq c_1 >0,$ then
$$ 0<\delta \,\mbox{meas}(\Om_{\delta}(u))\leq \int_{\Om_{\delta}(u)} |u| dx\leq c_2 ||u||_{X}, $$
which is a contradiction and therefore proves the claim. Now by \rm{(H3)}, for any given $\epsilon >0,\,\exists\,\,\delta>0$ such that
\begin{equation}\label{2}
 \frac{|f(\lambda,\,x,\,u)|}{|u|} \leq \epsilon\,\,\,\mbox{uniformly for }\,\,|u|<\delta.
\end{equation}
We split the integral in \eqref{1} on $\Om\backslash\Om_{\delta}(u)$\, and $\Om_{\delta}(u)$ and try to estimate the integrals:

\begin{equation} \label{3}
\begin{aligned}
 \int_{ \Om\backslash\Om_{\delta}(u)  } \frac{|f(\lambda,\,x,\,u)| |v | |\tilde{u}|}{  |u|}  & \leq \epsilon\int_{\Om\backslash\Om_{\delta}(u)} |\tilde{u}|  |v | \\
& \leq \epsilon ||  \tilde{u}  ||_{L^{2}(\Om\backslash\Om_{\delta}(u))}  ||v||_{L^{2}(\Om\backslash\Om_{\delta} (u) )}\\
& \leq c_3 \epsilon
\end{aligned}
\end{equation}
and
\begin{equation} \label{4}
\begin{aligned}
\int_{ \Om_{\delta}(u)  } \frac{|f(\lambda,\,x,\,u)| |v | |\tilde{u}|}{|u|}  & \leq \int_{\Om_{\delta}(u)} \frac{C(\lambda)|m_{1}(x)| |\tilde{u}||v|  } {|u|}    \\
& +   \frac{C(\lambda)}{||u||_X} \int_{\Om_{\delta}(u)} |m_{2}(x)| |u|^{\gamma} |v|.
\end{aligned}
\end{equation}
Let us compute each part in the RHS of \eqref{4} separately.

\begin{align}\nonumber
\int_{ \Om_{\delta}(u)  }  \frac{C(\lambda)|m_{1}(x)| |\tilde{u}||v|  } {|u|} & \leq c_4 \frac{1}{\delta}\int_{ \Om_{\delta}(u)  }  |m_{1}(x) |\tilde{u}||v|  \\ \nonumber
& \leq  c_4 \frac{1}{\delta}  ||m_1 ||_{L^{\frac{n}{2s}} (\Om_{\delta}(u))  } ||\tilde{u} ||_{L^{\frac{2n}{n-2s}} (\Om_{\delta}(u))}
|| v||_{L^{\frac{2n}{n-2s}} (\Om_{\delta}(u))}\\ \label{5}
& \longrightarrow 0\,\,\,\text{as}\,\,\text{meas}(\Om_{\delta}(u)) \longrightarrow 0,
\end{align}
in the last inequality, we have used the fact that $m_1\in L^{\frac{n}{2s}}$ and by Theorem\,\ref{em},\,$\tilde{u}, v\in L^{\frac{2n}{n-2s}}.$
Now

\begin{align}\nonumber
\frac{C(\lambda)}{||u||_X} \int_{\Om_{\delta}(u)} |m_{2}(x)| |u|^{\gamma} |v| &
\leq \frac{c_5}{||u||_X} \left(\int_{ \Om_{\delta}(u)}|u|^{\frac{2n}{n-2s}}\right)^{\frac{\gamma (n-2s)}{2n}  }
\left( \int_{ \Om_{\delta}(u)}|m_{2}(x)v(x)|^{\frac{2n}{2n-(n-2s)\gamma}  }   \right)^{\frac{2n-(n-2s)\gamma}{2n} }\\ \nonumber
& \leq \frac{c_5}{||u||_X}  ||m_2 ||_{L^{\gamma_1} (\Om_{\delta}(u)) }   ||u ||^{\gamma}_{L^{\frac{2n}{n-2s}} (\Om_{\delta}(u))}
|| v||_{L^{\frac{2n}{n-2s}}  (\Om_{\delta}(u))}\\ \nonumber
& \leq \frac{c_6}{||u||_X}||m_2 ||_{L^{\gamma_1} (\Om_{\delta}(u))} ||u ||^{\gamma}_{X}
|| v||_{X}\,\,\,\,\mbox{(by using Theorem\,\ref{em})}\\ \nonumber
& \leq c_6 ||m_2 ||_{L^{\gamma_1} (\Om_{\delta}(u))} ||u ||^{\gamma-1}_{X}
|| v||_{X}\\ \label{6}
& \longrightarrow 0\,\,\,\text{as}\,\,||u||_X\longrightarrow 0.
\end{align}
 This completes the proof of this lemma.
 \end{proof}
\begin{lem}\label{lm1}
 \rm{(i)}\, The operator $T$ is continuous and compact.\,\, \rm{(ii)}\, The operator $F$  is compact.
\end{lem}
\begin{proof}
\rm{(i)} It is easy to see the proof of this part. \\
\rm{(ii)} We show that the operator $F$  is compact.
Let $u_n\rightharpoonup u$ in $X.$  By Theorem\,\ref{em},  $u_n\rightarrow u$ in $L^p$ for $1<p< \frac{2n}{n-2s}.$
We claim that
$$ (F(\lambda,\,u_n),\,v  )_X \longrightarrow  (F(\lambda,\,u),\,v)_X  \,\,\,\mbox{as}\,\,n\rightarrow \infty.$$
Now by the continuity of the Nemytski operator $F$ from $L^{\frac{2n}{n-2s}}$ to  $ L^{\frac{2n}{n+2s}},$ it is easy to see that
$$ \sup_{||v||_X  \leq 1} \left| \int_{\Om} (f(\lambda,\,x,\,u_n)- f(\lambda,\,x,\,u))\ldotp v  \right |  \rightarrow 0\,\,\,\,\mbox{as}\,\,n\rightarrow \infty, $$
which proves the above claim.
\end{proof}

 Now we give the following definition:
 \begin{definition}
  We say that $\lambda \in \R$ and $u\in X$ solve the problem \eqref{bif} if
 \begin{equation}\label{td}
  S(u)- \lambda T(u) - F(\lambda,\,u)=0\,\,\mbox{in}\,\,X^*.
 \end{equation}
\end{definition}
 In the next lemma, we deal with an important inequality.
 \begin{lem}\label{io}
For $u,\,v\in X$ with $v>0$ a.e. in $\R^n,$ we have
 \begin{equation}\label{inq}
  \int_{\R^n}(-\Delta)^{\frac{s}{2}}v \ldotp (-\Delta)^{\frac{s}{2}}\left(\frac{u^2}{v}\right) dx \leq \int_{\R^n} |(-\Delta)^{\frac{s}{2} } u|^2  dx.
 \end{equation}
\end{lem}
 \begin{proof}
 We use the following simple inequality to establish \eqref{inq}:\\
 For any \,$a,\,b,\,c,\,d \in \R,\,\,c>0,\,d>0,$ we have
 \begin{equation}\label{si}
  (c-d) \left(\frac{a^2}{c}-  \frac{b^2}{d} \right)\leq (a-b)^2.
 \end{equation}
By definition, we have

\begin{align}\nonumber
 \int_{\R^n}(-\Delta)^{\frac{s}{2}}v\ldotp  (-\Delta)^{\frac{s}{2}}\left(\frac{u^2}{v}\right) dx  &
 =  \frac{C_{n,\,s}}{2} \int_{\R^n}\int_{\R^n} \frac{[v(x)-v(y)] \left(\frac {(u(x))^2}{v(x)} -\frac {(u(y))^2}{v(y)}\right)}{|x-y|^{n+2s}}  dxdy\\ \nonumber
 & \leq \frac{C_{n,\,s}}{2} \int_{\R^n}\int_{\R^n} \left(\frac {|(u(x) - u(y))|^2}{|x-y|^{n+2s}}\right) dxdy\,\,\,\,(\mbox{by}\,\,\eqref{si})\\ \label{7}
 &= \int_{\R^n} |(-\Delta)^{\frac{s}{2} } u|^2  dx.
\end{align}

 This completes the proof.
\end{proof}

\begin{prop}\label{pn1}
Let $2s<n<4s,\,\,0<s<1.$ Then any eigenfunction $v\in X$ associated to a positive eigenvalue $0<\lambda\neq \lambda_1$ changes sign. Moreover, if $N$ is a nodal domain of $v,$
then
\begin{equation}\label{nod}
 |N|\geq (c\lambda)^{-\frac{n}{2s}},
\end{equation}
where $c$ is some constant depending on $n$ and $|\,\,|$ denotes the Lebesgue measure of the set.
\end{prop}

\begin{proof}
Since $v$ is an eigenfunction associated with a positive eigenvalue $0<\lambda\neq \lambda_1,$
so we have
\begin{equation}\label{ww}
  (-\Delta)^{s} v= \lambda v\,\,\mbox{in}\,\,\Om,\,\,\,v=0\,\,\,\mbox{in}\,\,\,\R^n\backslash \Omega.
\end{equation}
We will prove this proposition by the method of contradiction. Let us assume that $v\geq 0$ in $\Om$ (the case $v\leq 0,$ can be dealt similarly).
By the strong maximum principle\,\cite{cab}, we have  $v>0$ in $\Omega.$ Let $\phi_1 >0$ be an eigenfunction of $ (-\Delta)^{s}$ associated with $\la_1.$
For any $\epsilon>0,$ it is easy to see that $\frac{\phi_1^2}{v+\epsilon}\in H_{0}^{s}(\R^n)$ and using $\frac{\phi_1^2}{v+\epsilon}\in H_{0}^{s}(\R^n)$ as a test function
in the weak formulation
of \eqref{ww}, we get
\begin{equation}\label{ca}
\int_{\R^n}(-\Delta)^{\frac{s}{2}}v\ldotp (-\Delta)^{\frac{s}{2}}\left( \frac{\phi_1^2}{v+\epsilon} \right)dx = \lambda \int_{\Om}v\ldotp \frac{\phi_1^2}{v+\epsilon} dx.
\end{equation}
Now from Lemma\,\ref{io}, (for the functions $\phi_1,\,v+\epsilon$), we have
\begin{equation}\label{inqa}
 \begin{aligned}
 \int_{\R^n} |(-\Delta)^{\frac{s}{2} } \phi_1|^2  dx & \geq   \int_{\R^n}(-\Delta)^{\frac{s}{2}}(v+\epsilon)\ldotp(-\Delta)^{\frac{s}{2}}\left(\frac{\phi_1^2}{v+\epsilon}\right) dx\\
 &= \int_{\R^n}(-\Delta)^{\frac{s}{2}}v \ldotp (-\Delta)^{\frac{s}{2}}\left(\frac{\phi_1^2}{v+\epsilon}\right) dx.
\end{aligned}
\end{equation}
Since $\phi_1 >0$ is an eigenfunction of $(-\Delta)^{s}$ associated with $\lambda_1$ so this implies that
\begin{equation}\label{wi}
 \int_{\R^n}(-\Delta)^{\frac{s}{2}}\phi_1\ldotp (-\Delta)^{\frac{s}{2}}\psi dx= \lambda_1 \int_{\Om}\phi_1 \ldotp \psi dx,\,\,\,\forall\,\,\psi\in X.
\end{equation}
Now on taking $\psi= \phi_1$ in \eqref{wi}, we obtain
 \begin{equation}\label{wii}
 \int_{\R^n}|(-\Delta)^{\frac{s}{2}}\phi_1|^2 dx= \int_{\Om} \lambda_1 \phi_1^2 dx.
\end{equation}
From \eqref{ca}, \eqref{inqa} and \eqref{wii}, we get
\begin{equation} \label{8}
\begin{aligned}
 0 & \leq  \int_{\R^n}|(-\Delta)^{\frac{s}{2}}\phi_1|^2 dx -  \int_{\R^n}(-\Delta)^{\frac{s}{2}}v \ldotp (-\Delta)^{\frac{s}{2}}\left(\frac{\phi_1^2}{v+\epsilon}\right) dx\\                             &
 =  \int_{\Om} \lambda_1 \phi_1^2 dx- \lambda \int_{\Om}v\ldotp \frac{\phi_1^2}{v+\epsilon} dx\\
& = \int_{\Om} \phi_1^2(\lambda_1  - \lambda)dx,\,\,\,\mbox{as}\,\,\,\epsilon \rightarrow 0\\
& <0\,\,(\mbox{since}\,\,\lambda_1<\lambda),
\end{aligned}
\end{equation}
which yields a contradiction. Now we prove the estimate \eqref{nod}. Assume that $v>0$ in $N,$ the case $v<0$ can be dealt similarly.
Since $2s<n<4s,\,\,0<s<1,$  so by \,\cite{cs2,ros,sil}, we have $v\in C(\R^n)\cap X.$ Then $v\lvert_{N}\in H_{0}^{s}(N)$ and therefore the function $\eta$
defined as
\[\eta(x)=\begin{cases}
v(x) & x\in N\\
0 & x\in \R^n\backslash N
\end{cases}
\]

 belongs to $X=H_{0}^{s}(\R^n).$ Now using $\eta$ as a test function in the weak formulation of \eqref{ww}, we get
 \begin{equation*}
 \int_{\R^n}(-\Delta)^{\frac{s}{2}}v \ldotp (-\Delta)^{\frac{s}{2}}\eta dx= \lambda \int_{\Om}v\ldotp \eta dx,\,\,\,\forall\,\,\eta\in X.
 \end{equation*}
This implies that
\begin{equation} \label{9}
\begin{aligned}
 \int_{N}|(-\Delta)^{\frac{s}{2}}v|^2 dx & =  \lambda \int_{N} v^2  dx\\
&  \leq  \left(\int_{N} v^{\frac{2n}{n-2s} }dx\right)^{\frac{n-2s}{n} } |N|^{\frac{2s}{n} } \\
& = \lambda  ||v||^{2}_{L^{\frac{2n}{n-2s}   }    }  |N|^{\frac{2s}{n} } \\
& \leq c \lambda || v||^{2}_{H_{0}^{s}(N)}  |N|^{\frac{2s}{n} }\,\,\,\,\mbox{(by Theorem\,\ref{em}, where}\, \,c\,\,\,\mbox{is an embedding constant}), \\
&= c \lambda \left(\int_{N}|(-\Delta)^{\frac{s}{2}}v|^2 dx\right)  |N|^{\frac{2s}{n} }.
\end{aligned}
\end{equation}
 So from the last inequality, it implies that
$$ |N|\geq (c\la)^{-\frac{n}{2s}},              $$
which proves the estimate.
\end{proof}

\begin{prop}\label{pn2}
 Let $2s<n<4s,\,\,0<s<1.$ Then $\lambda_1$ is isolated, that is, there exists $\delta>0$ such that there are no other eigenvalues of \eqref{ei} in the interval $(\lambda_1,\,\lambda_1+ \delta).$
 \end{prop}
\begin{proof}
We will prove this proposition by the method of contradiction. Suppose that there exists a sequence of eigenvalues $\lambda_n$ of \eqref{ei} with $0<\lambda_n\searrow \lambda_1.$
Let $u_n$ be a sequence of eigenfunctions associated with $\lambda_n,$ i.e,
\begin{eqnarray}\label{sei}
 \left\{
\begin{gathered}
 (-\Delta)^{s} u_n  =  \lambda_n u_n\,\,\,\mbox{in}\,\,\Omega,\\
 u_n   = 0 \,\,\,\,\mbox{in}\,\,\,\,\,\,\,\R^n \backslash\Omega.
\end{gathered}\right.
 \end{eqnarray}
 Since $$ 0< \int_{\R^n} |(-\Delta)^{\frac{s}{2}}u_n|^2 dx= \lambda_n \int_{\Om} u_n^2 dx,$$
 we can define
 $$v_n: = \frac{u_n}{ (\int_{\Om} u^{2}_{n}dx)^{\frac{1}{2} }  }.    $$
 It is easy to see that $v_n$ is bounded in $X$ so there exists a subsequence (we denote it by $v_n$) and $v\in X$
 such that
\begin{eqnarray*}
v_{n}& \rightharpoonup &  v\,\,\,\,\mbox{in}\,\,X. \\
v_n & \longrightarrow &   v \,\,\,\,\mbox{in}\,\,L^{p},\,\,p\in [1,\,2_{*}(s)).\\
v_{n}(x)& \longrightarrow &  v(x)\,\,\,\,\mbox{a.e.}\,\,\,x\in \Om.
\end{eqnarray*}
 It is easy to see that $\int_{\Om} v^2 dx=1$ and
 $$ \int_{\R^n} |(-\Delta)^{\frac{s}{2}}v|^2 dx \leq \liminf_{n\rightarrow \infty}\int_{\R^n}  |(-\Delta)^{\frac{s}{2}}v_n|^2 dx = \lambda_1 $$
 and therefore $\lambda_1= \int_{\R^n} |(-\Delta)^{\frac{s}{2}}v|^2 dx.$ This implies that $v$ is an eigenfunction associated with $\lambda_1.$
 Now one can see that $|v|$ is also an eigenfunction associated with $\lambda_1$ and by strong maximum principle\,\cite{cab}, $|v|>0$ in $\Om.$
 This implies that either $v>0$ or $v<0.$ Let us take $v>0$ (the proof in the case $v<0$ is dealt similarly). Let $\Om^{-}_n = \{x\in \Om:  v_{n}(x)<0 \}.$
 Now by the Egorov's theorem, $v_n\rightarrow v$ uniformly on $\Om$ with the exception of the set of arbitrarily small measure. This implies that
 \begin{equation}
  |\Om^{-}_n|\longrightarrow 0\,\,\,\,\mbox{as}\,\,n\longrightarrow \infty,
 \end{equation}
but this contradicts to the estimate \eqref{nod} for $\Om^{-}_n$ and hence the proof is complete.
\end{proof}

 \section{Topological degree and its properties}
 In this section, we define the topological degree for the operators given by the left-hand side of \eqref{td}. Let us recall some basic results on the degree theory for operators
 from a Banach space $V$ to $V^*.$  Let $V$ be a real reflexive Banach space and $V^*$ its dual, and $A:V\longrightarrow V^*$ be a demicontinuous operator.  We assume that
 $A$ satisfies the condition $\alpha(V),$ i.e.,\\
 for any sequence $u_n \in V$ satisfying $u_n \rightharpoonup u_0$ weakly in $V$ and
 $$ \limsup (A(u_n),\,u_{n}-u_0)_V \leq 0,                      $$
 then $u_n \longrightarrow u_0$ strongly in $V.$
 Then one can define the degree Deg $[A;U,0 ],$ where $U\subset V$ is a bounded open set such that $A(u)\neq 0$ for any $u\in \partial U.$
For its properties, we refer the reader to \cite{sk}. Let us recall some basic definitions.
\begin{definition}
A point $v_0\in V$ is called a critical point of $A$ if $A(v_0)=0.$
\end{definition}
\begin{definition}
We say that $v_0$
 is an isolated critical point of $A$ if there exists an $\epsilon >0$ such that for any $v\in B_{\epsilon}(v_0),\,\,A(v)\neq 0$ if $v\neq v_0.$
\end{definition}
 In this case, one can see that the limit
$$   \mbox{Ind}(A,\,v_0)= \lim_{\epsilon\rightarrow 0} \mbox{Deg} [A; B_{\epsilon}(v_0), 0]                                     $$
 exists and is called the index of the isolated critical points $v_0$ of $A.$ Let us assume that $A$ is a potential operator, i.e., for
 some continuously differentiable functionals $\Psi:V\longrightarrow \R,\,\,\Psi'(v)=A(v),\,v\in V.$ Let us state two lemmas which have been used in the proof of the next theorem.
The proof of these  lemmas can be seen in \cite{sk}.

 \begin{lem}\cite{sk}\label{sk1}
  Let $v_0$ be a local minimum of $\Psi$ and an isolated critical point of $A.$ Then
  $$\mbox{Ind}(A,\,v_0)=1.$$
 \end{lem}
\begin{lem}\cite{sk}\label{sk2}
 Assume that $(A(v), v)_{V} >0$ for all $v\in V,\,\,||v||_{V}= r.$ Then
 $$ \text{Deg}[A,\,B_{r}(0),0] =1.$$
\end{lem}

 \begin{rem}\label{rm1}
 It is easy to see that every continuous map $A:V\longrightarrow V^{*}$ is also demicontinuous. Now one can also see that if $A$  satisfies the condition $\alpha(V)$
 and for any compact operator $K:V\longrightarrow V^{*},\,\,\,A+K$ also satisfies the condition $\alpha(V).$
 \end{rem}

 \begin{lem}\label{lm2}
 The operator $S: X \longrightarrow X^*$ satisfies $\alpha(X).$
 \end{lem}
\begin{proof}
 Let $u_n\rightharpoonup u_0$ in $X$ and
$$ \limsup (S(u_n),\,u_{n}-u_0)_X \leq 0,$$
then we claim $u_n \longrightarrow u_0$ in $X.$
 \begin{equation} \label{cat}
\begin{aligned}
 0 & \geq  \limsup_{n\rightarrow \infty} (S(u_n),\,u_{n}-u_0)_X  \\
& =   \limsup_{n\rightarrow \infty} (S(u_n)- S(u_0),\,u_{n}-u_0)_X, \,\,\,(\mbox{by the definition of}\,\,u_n\rightharpoonup u_0\,\,\mbox{ in}\,\,X) \\
& = \lim_{n\rightarrow \infty}  \int_{\R^n}| (-\Delta)^{\frac{s}{2} }(u_{n} - u_0)  |^2  dx       \\
& \geq 0,
\end{aligned}
\end{equation}
 and therefore the claim is proved.
\end{proof}
Now in view of Remark\,\ref{rm1}, from Lemma\,\ref{lm1} and Lemma\,\ref{lm2}, it is trivial to see that the operator $S_{\lambda}:=S-\lambda T- F(\lambda,\,\ldotp)$ satisfies the
condition $\al(X)$
and therefore we can define the Deg$[S-\lambda T- F(\lambda,\,\ldotp); U, 0],$ where $U\subset X$ is a bounded open set such that $T_{\lambda}(u)\neq 0$ for any $u\in \partial U.$

 \section{Proof of main theorems: bifurcation from $\lambda_1$}
 \begin{definition}
  Let $W= \R\times X$ be equipped with the norm
  \begin{equation}\label{no}
   ||(\lambda,\,u)||_{W} = (|\lambda|^2 + ||u||_{X}^2)^{\frac{1}{2}},\,\,\,\,(\lambda,\,u)\in W.
  \end{equation}
Let us define
$$ \C= \{(\lambda,\,u)\in W: (\lambda,\,u) \,\,\mbox{solves}\,\,\eqref{bif},\,\,u\neq 0                \}. $$
 We say that $\C$ is a continuum of nontrivial solutions of \eqref{bif} if it is a connected set in $W$ with respect to the topology induced by the norm \eqref{no}, see\,\cite{rabh}.
 Following in the sense of Rabinowitz, we say that $\lambda_0 \in \R$
 is a bifurcation point of \eqref{bif} if there is a continuum of nontrivial solutions  $\C$ of \eqref{bif} such that $(\lambda_0,\,0)\in \overline{\C}$
 and either $\C$ is unbounded in $W$ or there
 is an eigenvalue $\lambda^* \neq \lambda_0$ such that $(\lambda^*,\,0)\in \overline{\C}.$
 \end{definition}

\subsection{Proof of Theorem 1.1}

 \begin{proof}
 The proof has the same spirit as in \cite{dra}. The proof consists of three steps.

\textit{Step 1.} Let us consider the operator $\tilde{S}_\lambda(u)=S(u)-\lambda T(u).$
From the variational characterization of $\lambda_1,$ it follows that for $\lambda\in (0,\lambda_1)$ and any $u\in X$ with $\|u\|_X \neq 0,$ we have
\be \nonumber
(\tilde{S}_\lambda(u),u)_X>0.
\ee
Then the degree
\be\label{inter1}
\text{Deg }[\tilde{S}_\lambda; B_r(0), 0]
\ee
is well defined for any $\lambda\in (0,\lambda_1)$ and any ball $B_r(0)\subseteq X.$ Applying Lemma \ref{sk2}, we get
\be\label{inter2}
\text{Deg } [\tilde{S}_\lambda; B_r(0),0]=1, \,\, \lambda\in (0,\lambda_1).
\ee
Now by Proposition\,\ref{pn2}, there exists a $\delta>0$ such that the interval $(\lambda_1, \lambda_1+\delta)$ does not contain any eigenvalue of \eqref{ei} and
hence the degree \eqref{inter1} is well defined also for $\lambda\in (\lambda_1, \lambda_1+\delta)$. To evaluate $\text{Ind}(\tilde{S}_\lambda,\,0)$
for $\lambda \in (\lambda_1, \lambda_1+\delta)$, we use the similar procedure as in \cite{dra1,dra2}.

Fix a $K>0$ and define a function $h\colon \R\rar \R$ by
\be\nonumber
h(t)=
\left\{
  \begin{array}{ll}
    0, & \mbox{for } t\leq K, \\
    \frac{2\delta}{\lambda_1}(t-2K), & \mbox{for } t\geq 3K,
  \end{array}
\right.
\ee
and $h(t)$ is  positive and strictly convex in $(K,\,3K)$. We define the functional
\be \nonumber
J_\lambda(u)=\frac{1}{2}(S(u),u)_X -\frac{\lambda}{2} (T(u),u)_X + h \left(\frac{1}{2}(S(u),u)_X\right).
\ee
Then $J_\lambda$ is continuously Fr\'{e}chet differentiable and its critical point $u_0\in X$ corresponds to the solution of the equation
\be \nonumber
S(u_0)-\frac{\lambda}{1+h'\left(\frac{1}{2}(S(u_0),u_0)_X\right)}T(u_0)=0.
\ee
However, since $\lambda \in (\lambda_1, \lambda_1+\delta)$, the only nontrivial critical points of $J_\lambda'$ occur if
\be\nonumber
h'\left(\frac{1}{2}(S(u_0),u_0)\right)=\frac{\lambda}{\lambda_1}-1,
\ee
 and hence we must have
\be \nonumber
\frac{1}{2}(S(u_0),u_0)_X \in (K,3K).
\ee
In this case, either we have $u_0=-u_1$ or $u_0=u_1,$ where $u_1>0$ is the principle eigenfunction. So, for $\lambda\in (\lambda_1, \lambda_1+\delta), \,\,J_\lambda'$ has precisely
three isolated critical  points $-u_1,\,0,\,u_1.$ Now it is easy to see that the functional $J_\lambda$ is weakly lower semicontinuous.

Indeed, assume $u_n\rar u_0$ weakly in $X.$ Then
\be \label{inter3}
(T(u_n),u_n)_V \rar (T(u_0),u_0)_X,
\ee
due to compactness of $T,$ and
\begin{gather}\nonumber
\liminf_{n\rar \infty} \left(\frac{1}{2} (S(u_n),u_n)_X + h \left(\frac{1}{2}(S(u_n), u_n)_X\right)\right)\\
\geq \frac{1}{2} (S(u_0),u_0)_X + h \left(\frac{1}{2}(S(u_0), u_0)_X\right)\label{inter4}
\end{gather}
by the facts that $\liminf \|u_n\|_X \geq \|u_0\|_X$ and $h$ is nondecreasing.

The relations \eqref{inter3} and \eqref{inter4} then imply
\be \nonumber
\liminf_{n\rar \infty} J_\lambda(u_n) \geq J_\lambda(u_0).
\ee
Observe that $J$ is coercive, that is,
\be \nonumber
\lim_{\|u\|_X\rar \infty}J_\lambda (u)=\infty.
\ee
In fact, we can estimate $J_\lambda(u)$ as follows:
 \begin{align*}
 J_\lambda(u) &= \frac{1}{2} (S(u),u)_X-\frac{\lambda_1}{2} (T(u),u)_X+ \frac{\lambda_1 -\lambda}{2} (T(u),u)_X+ h\left(\frac{1}{2} (S(u),u)_X\right)\\
 & \geq \frac{\lambda_1 -\lambda}{2} (T(u),u)_X+ h\left(\frac{1}{2} (S(u),u)_X\right) \\
 & \geq\frac{\lambda_1 -\lambda}{2\lambda_1} (S(u),u)_X+ h\left(\frac{1}{2} (S(u),u)_X\right) \\
 & \geq -\frac{\delta}{2\lambda_1}(S(u),u)_X + \frac{2\delta}{\lambda_1}\left[\frac{1}{2}(S(u),u)_X-2K\right]\rar \infty,
 \end{align*}
 for $\|u\|_X\rar \infty$. Here we have used the variational characterization for $\lambda_1$ and the definition of $h.$
 Since $J_\lambda$ is an even functional, there are precisely
 two points at which the minimum of $J_\lambda$ is achieved: $-u_1, u_1.$
 The point $0$ is saddle type of isolated critical point. By
 Lemma \ref{sk1}, we have
 \be\label{inter5}
  \text{Ind } (J_\lambda', -u_1)=  \text{Ind } (J_\lambda', u_1)=1.
 \ee
 Also, we have that $(J_\lambda'(u),u)_X>0$ for any $u\in X,\, \|u\|_X=R$ with $R>0$ large enough. Now in order to show the coercivity of
 $J_\lambda,$  we use the following estimate:
 \begin{align*}
 (J_\lambda'(u),u)_X & = (S(u),u)_X-\lambda (T(u),u)_X+ h'\left(\frac{1}{2}(S(u),u)_X \right)(S(u),u)_X \\
 & = (S(u),u)_X-\lambda_1 (T(u),u)_X \\
 &\,\,\,\, + h'\left(\frac{1}{2}(S(u),u)_X \right)\left[(S(u),u)_X -\frac{\lambda-\lambda_1}{h'\left(\frac{1}{2}(S(u),u)_X \right)}(T(u),u)_X\right] \\
 & \geq \frac{2\delta}{\lambda_1} \left[(S(u),u)_X-2K\right]\cdot \left[(S(u),u)_X-\frac{\lambda_1}{2} (T(u),u)_X\right]\rar\infty
 \end{align*}
 for $\|u\|_X\rar \infty$. We have used again the variational characterization of $\lambda_1$ and the definition of $h.$ Lemma \ref{sk2} then implies that
 \be \label{inter6}
 \text{Deg } [J_\lambda'; B_R(0),0]=1.
 \ee
 We choose $R$ so large that $\pm u_1 \in B_R (0).$ Now, by the additivity of the degree (see \cite{sk}), and \eqref{inter5} and \eqref{inter6}, we have
 \be \label{inter7}
 \text{Ind } (J_\lambda', 0)=-1.
 \ee
 Further, by the definition of $h,$
 \be \label{inter8}
 \text{Deg }[\tilde{S}_\lambda; B_{r^*}(0), 0]= \text{Ind } (J_\lambda',0)
 \ee
 for $r^*>0$ small enough. We then conclude from \eqref{inter2}, \eqref{inter7} and \eqref{inter8} that
 \be \label{inter9}
   \begin{array}{ll}
     \text{Ind }(\tilde{S}_\lambda,0)=1, & \lambda\in(0,\lambda_1), \\
     \text{Ind }(\tilde{S}_\lambda,0)=-1, & \lambda\in (\lambda_1,\lambda_1+\delta).
   \end{array}
\ee
\textit{Step 2.} It follows from \eqref{li} and the homotopy invariance of the degree that for $r^*>0$ small enough,
 \be \nonumber
 \text{Deg }[S_\lambda; B_{r^*}(0),0]=\text{Deg } [\tilde{S}_\lambda; B_{r^*}(0),0]
 \ee
 for $\lambda\in (0,\lambda_1+\delta)\backslash \{\lambda_1\}$. We have, from \eqref{inter9},
\be \nonumber
   \begin{array}{ll}
     \text{Ind }(S_\lambda,0)=1, & \lambda\in(0,\lambda_1), \\
     \text{Ind }(S_\lambda,0)=-1, & \lambda\in (\lambda_1,\lambda_1 +\delta).
   \end{array}
\ee
\textit{Step 3.} Now thanks to Theorem 1.3\,\cite{rab},  we use the similar lines of proof as Theorem 1.3\,\cite{rab} and draw the conclusion of this theorem.
This completes the proof.

\end{proof}

\subsection{Proof of Theorem 1.2}
\begin{proof}
Let us recall that the problem
\begin{eqnarray}\label{ae}
 \left\{
\begin{gathered}
 (-\Delta)^{s} u  =  h\,\,\,\mbox{in}\,\,\Omega,\\
 u   = 0 \,\,\,\,\mbox{in}\,\,\,\,\,\,\,\R^n \backslash\Omega
\end{gathered}\right.
 \end{eqnarray}
has a unique solution  for each $h\in H^{-s} (\Omega),$ i.e., there exists unique $u\in H_{0}^{s}(\Omega)=X_1$ such that
\begin{equation}
 \int_{\R^n} (-\Delta)^{\frac{s}{2}}u\ldotp (-\Delta)^{\frac{s}{2}}v = \int_{\Om} h\ldotp v,\,\,\forall\,\,v\in X_1.
\end{equation}
Let us denote by $R(h)$ the unique weak solution of \eqref{ae}. Then $R:\,\,X_1^* \longrightarrow  X_1$ is a continuous operator. Since we know that
$X_1$ embeds compactly into $L^{r}(\Om)$ for each $r\in [1,\,2_{*}(s)).$  Therefore, it follows that the restriction of $R$ to $L^{r'}(\Om)$ is a completely
continuous operator and thus $R$ transforms weak convergence in $L^{r'}(\Om)$ into strong convergence in $X.$  Let us reformulate the problem \eqref{bif}. It is clear
that $(\lambda,\,u)$ is a solution of \eqref{bif} iff $(\lambda,\,u)$ satisfies
\begin{equation}\label{refo}
 u= R(\lambda u + F(\lambda,\,u) ),
\end{equation}
where $F(\lambda,\,\ldotp)$ denotes the Nemitsky operator associated with $f.$ From \rm{(H4)}, it is easy to see that the right hand side of \eqref{refo} defines a completely
continuous operator from $X$ into itself. Let us suppose that $(\bar{\lambda},\,0)$ is a bifurcation point of \eqref{bif} and we show that
$\bar{\lambda}$ is an eigenvalue of \eqref{ei}. Since $(\bar{\lambda},\,0)$ is a bifurcation point, there is a sequence $\{ \lambda_n,\,u_n\}$  of nontrivial solutions of \eqref{bif}
such that $\lambda_n\longrightarrow \bar{\lambda}$ in $\R$ and $u_n \longrightarrow 0$ in $X.$ Also, since
\begin{eqnarray}\label{ab}
 \left\{
\begin{gathered}
 (-\Delta)^{s} u_n  = \lambda_n u_n + f(\lambda_n,\,x,\,u_n) \,\,\,\mbox{in}\,\,\Omega,\\
 u_n   = 0 \,\,\,\,\mbox{in}\,\,\,\,\,\,\,\R^n \backslash\Omega,
\end{gathered}\right.
 \end{eqnarray}
so we have
\begin{equation}\label{abc}
 \widetilde{u_n} = R\left(\lambda_n \widetilde{u_n} + \frac{F(\lambda_n,\,u_n)}{||u_n||_X}\right),
\end{equation}
where $\widetilde{u_n}= \frac{u_n}{||u_n||_X}.$ Now we claim that
\begin{equation}\label{cn}
 \frac{F(\lambda_n,\,u_n)}{||u_n||_X}\longrightarrow 0\,\,\mbox{in}\,\,L^{q'},
\end{equation}
where $q$ is chosen in $\rm{(H4)}.$ We note that

\begin{equation}\label{cn1}
 \frac{F(\lambda_n,\,u_n)}{||u_n||_X}=    \frac{F(\lambda_n,\,u_n)}{u_n} \widetilde{u_n}.
\end{equation}
 In order to prove the claim, it is sufficient to find a real number $r>1$ and a constant $C>0$ such that
 \begin{equation}\label{cn2}
 \left|\frac{F(\lambda_n,\,u_n)}{u_n}\right|^{q'}\longrightarrow 0\,\, \mbox{in}\,\,L^{r}
\end{equation}
and
\begin{equation}\label{cn3}
|| |\widetilde{u_n}|^{q'}    ||_{L^{r'}} \leq C,\,\,\,\forall\,\,n\in N.
\end{equation}
Indeed, from \eqref{cn1} and using H\"{o}lder's inequality, we find such an $r.$ Let us fix $\epsilon>0$ and choose positive numbers $\delta= \delta(\epsilon)$ and $M = M(\delta)$
such that for every $x\in \Om$ and $n\in N,$ in view of \rm{(H3)}, \rm{(H4)}, the following inequalities hold:
 \begin{eqnarray}\label{i}
 \begin{gathered}
 |f(\lambda_n,\,x,\,t)|\leq \epsilon |t| \,\,\mbox{for}\,\,|t|\leq \delta.\\
 |f(\lambda_n,\,x,\,t)|\leq M |t|^{q-1} \,\,\mbox{for}\,\,|t|\geq \delta.
 \end{gathered}
 \end{eqnarray}
Let $r$ be a real number greater than $1.$  Then from \eqref{i}, we obtain

\begin{equation} \label{cnn}
\begin{aligned}
\left \|\left|\frac{F(\lambda_n,\,u_n)}{u_n}\right|^{q'}\right\|^{r}_{L^r}
&  = \int_{\Om} \left|\frac{f(\lambda_n,\,x,\,u_n) }{u_n}  \right|^{r q'} dx       \\
& = \epsilon |\Om|   +   M^{q' r} \int_{\Om} |u_n|^{q' r (q-2) }.  \\
\end{aligned}
\end{equation}
 From \eqref{cnn} and since $u_n\longrightarrow 0$ in $X,$ we have that \eqref{cn2} is satisfied if
 \begin{equation}\label{fi1}
  q' r (q-2) < 2_{*} (s).
 \end{equation}
Since $\widetilde{u_n}$ is bounded in $L^{2_{*}(s)}, $ we see that \eqref{cn3} is satisfied if

 \begin{equation}\label{fi2}
  q' r'  < 2_{*} (s).
 \end{equation}
 Finding an $r$ satisfying \eqref{fi1} and \eqref{fi2} is equivalent to find an $r$ such that
 $$  \frac{q'(q-2)}{2_{*}(s)} < \frac{1}{r} < 1- \frac{q'}{2_{*}(s)} $$
 and this is always possible because $q< 2_{*}(s).$ Now from \eqref{abc} and \eqref{cn} and the compactness of $R$, we can assume  (passing to a subsequence, if necessary)
 $\widetilde{u_n}\longrightarrow \widetilde{u}$ in $X.$ Now passing the limit in \eqref{abc}, we obtain that
 $$  \widetilde{u}= R(\bar{\lambda} \widetilde{u}  ).$$  Since $\widetilde{u} \neq 0$ so this implies that $\bar{\lambda}$ is an eigenvalue of \eqref{ei} and thus the theorem is proved.

\end{proof}



\begin{thebibliography}{0}
\bibitem{ap} D. Applebaum, L\'{e}vy  processes-from probability to finance and quantum groups, Notices Amer. Math. Soc. \textbf{51},
2004, 1336--1347.

\bibitem{ben} D. A. Benson, S.W. Wheatcraft, and M. M. Meerschaert, Application of a fractional
advection-dispersion equation, Water Resources Res. \textbf{36}, 2000, 1403--1412.

\bibitem{be} J. Bertoin, L\'{e}vy Processes, Cambridge Tracts in Math., Vol. \textbf{121}, Cambridge Univ. Press, Cambridge, 1996.

\bibitem{bis}  G. M. Bisci, R. Servadei, A bifurcation result for nonlocal fractional equations,  Anal. Appl.\textbf{ 13} (4) 2015, 371--394.

\bibitem{bran} C. Br\"{a}ndle, E. Colorado, A. de Pablo, U. S\'{a}nchez, A concave-convex elliptic problem involving the fractional Laplacian,
Proc. Roy. Soc. Edinburgh Sect A \textbf{143}(1), 2013, 39--71.

\bibitem{beq} J. Busca, Maria J. Esteban, A.Quaas, Nonlinear eigenvalues and bifurcation problems for Pucci’s operators, Ann. I. H. Poincar\'{e} \textbf{22}, 2005,  187--206.

\bibitem{cab} X. Cabr\'{e}, Y. Sire,  Nonlinear equations for fractional Laplacians, I: Regularity,
maximum principles, and Hamiltonian estimates, Ann. I. H. Poincar\'{e} \textbf{31}, 2014,  23--53.

diffusion fractionnaire, C. R. Math. Acad. Sci. Paris \textbf{347} (23-24), 2009, 1361--1366.


\bibitem{cs} L.\,Caffarelli and L.\,Silvestre, An extension problem related
to the fractional Laplacian, Comm. Partial Differential Equations \textbf{32}, 2007, 1245--1260.

\bibitem{cs1} L. Caffarelli, J.M. Roquejoffre, Y. Sire, Variational problems for free boundaries for the fractional Laplacian, J. Eur.
Math. Soc. \textbf{12}, 2010, 1151--1179.

\bibitem{cs2}  L. Caffarelli, S. Salsa, L. Silvestre, Regularity estimates for the solution and the free boundary of the obstacle
problem for the fractional Laplacian, Invent. Math. \textbf{171}, 2008, 425--461.


\bibitem{cap} A. Capella, J. D\'{a}vila, L. Dupaigne, and Y. Sire, Regularity of radial extremal solutions for some
non-local semilinear equations, Comm. Partial Differential Equations \textbf{36} (8), 2011, 1353--1384.


\bibitem{cha} S.Y.A. Chang, M. Gonz\'{a}lez, Fractional Laplacian in conformal geometry, Adv. Math. \textbf{226}, 2011, 1410--1432.

\bibitem{del} M.A. Del Pino and R. Man\'{a}sevich, Global bifurcation from the eigenvalues of the $p$-Laplacian, J. Diff. Equations \textbf{92}, 1991, 226--251.


\bibitem{di}  E. Di Nezza, G. Palatucci, E. Valdinoci, Hitchhiker's guide to the fractional Sobolev spaces, Bull. Sci. Math. \textbf{136}, 2012, no. 5, 521--573.

\bibitem{dra} P. Dr\'{a}bek, Y.X.Huang, Bifurcation problems for the $p$-Laplacian in $\R^N,$ Trans. Amer. Math.Soc. \textbf{349} (1),  1997, 171--188.

\bibitem{dra1} P. Dr\'{a}bek,  On the global bifurcation for a class of degenerate equations, Ann. Mat. Pura. Appl. \textbf{159},  1991, 1--16.

\bibitem{dra2} P. Dr\'{a}bek, Solvability and bifurcations of nonlinear equations, Pitman Research Notes in Math. \textbf{264}, Longman, Harlow, 1992.

\bibitem{ga}  A. Garroni, S. M\"{u}ller, $\Gamma$-limit of a phase-field model of dislocations, SIAM J. Math. Anal. \textbf{36}, 2005, 1943--1964.


\bibitem{mar} Maria J. Esteban, P. Felmer, A. Quaas, Eigenvalues for radially symmetric fully nonlinear operators, Comm. Partial Differential Equations \textbf{35}(9),
2010,  1716--1737.


\bibitem{mus} R. Musina, A.I.Nazarov, On fractional Laplacians, Comm. Partial Differential \textbf{39}(9), 2014,  1780--1790.

\bibitem{per}  K. Perera, M. Squassina, and Y. Yang, Bifurcation and multiplicity results for critical fractional $p$-laplacian problems, Mathematische Nachrichten  \textbf{289},
2016, no. 2-3, 332--342.

\bibitem{rabh}  P.H. Rabinowitz, Some aspect of nonlinear eigenvalue problem, Rocky Mountain J. Math. \textbf{74}, 1973, 161--202.


\bibitem{rab} P.H. Rabinowitz, Some global results for nonlinear eigenvalue problems,  J.\,Func. Anal. \textbf{7}, 1971, 487--513.


\bibitem{ros} X. Ros-Oton,  J. Serra, The extremal solution for the fractional Laplacian, Calc. Var. Partial Differential Equations\,\textbf{50}, 2014, 723--750.

\bibitem{ram} A.J. Rumbos, A.L.Edelson, Bifurcation properties of semilinear elliptic equations in $R^n,$ Differential Integral Equations \textbf{7}(2), 1994,  399--410.

\bibitem{sa}  M. Saxton, Anomalous subdiffusion in fluorescence photobleaching recovery: A Monte Carlo study, J. Biophys \textbf{81}, 2001, 2226--2240.







\bibitem{sil} L. Silvestre, Regularity of the obstacle problem for a fractional power of the Laplace operator, Comm. Pure Appl. Math. \textbf{60} (1), 2007, 67--112.


\bibitem{sir} Y. Sire and E. Valdinoci, Fractional Laplacian phase transitions and boundary reactions: a
geometric inequality and a symmetry result, J. Funct. Anal. \textbf{256}(6), 2009, 1842--1864.

\bibitem{sk} I.V.Skrypnik, Methods for analysis of nonlinear elliptic boundary value problems,
Vol. 139, Translations of mathematical monographs, American Mathematical Society, 1994.


\bibitem{ta1}  J. Tan, The Brezis-Nirenberg type problem involving the square root of the Laplacian, Calc. Var. Partial Differential Equations \textbf{42}(1-2), 2011, 21--41.

\bibitem{ta2} J. Tan, Y. Wang, and J. Yang, Nonlinear fractional field equations, Nonlinear Anal. \textbf{75}(4), 2012, 2098--2110.




\bibitem{vaz} J. L. V\'{a}zquez, Recent progress in the theory of nonlinear diffusion with fractional Laplacian operators, Discrete Contin. Dyn. Syst. Ser. S \,
\textbf{7} (4), 2014, 857--885.


\end{thebibliography}
\end{document}